\documentclass[11pt]{article}
\usepackage[T1]{fontenc}
\usepackage[utf8]{inputenc}
\usepackage{lmodern}
\usepackage[a4paper,margin=28mm]{geometry}
\usepackage{amsmath,amssymb,amsthm,mathtools}
\usepackage{microtype}
\usepackage[numbers,sort&compress]{natbib}
\usepackage{hyperref}
\hypersetup{hidelinks,pdftitle={A lifted entropy splitting lemma and applications},
  pdfauthor={Gabriel Coutinho and Luis Higino}}
\numberwithin{equation}{section}
\newtheorem{theorem}{Theorem}[section]
\newtheorem{lemma}[theorem]{Lemma}
\newtheorem{proposition}[theorem]{Proposition}
\newtheorem{corollary}[theorem]{Corollary}
\theoremstyle{remark}

\newcommand{\R}{\mathbb R}
\newcommand{\C}{\mathbb C}
\newcommand{\Smat}{\mathbb S}
\newcommand{\one}{\mathbf 1}
\newcommand{\ThetaBody}{\operatorname{TH}}
\newcommand{\abl}{\operatorname{abl}}
\newcommand{\diag}{\operatorname{diag}}
\newcommand{\Diag}{\operatorname{Diag}}
\newcommand{\tr}{\operatorname{tr}}

\newcommand{\ip}[2]{\langle #1,#2\rangle}
\newcommand{\cart}{\mathbin{\square}}
\newcommand{\cK}{\mathcal K}
\newcommand{\cC}{\mathcal C}
\newcommand{\cE}{\mathcal E}
\newcommand{\Psd}{\operatorname{Psd}}
\newcommand{\supp}{\operatorname{supp}}

\title{A lifted entropy splitting lemma and applications}
\author{Gabriel Coutinho\thanks{Departamento de Ci\^encia da Computa\c{c}\~ao,
Universidade Federal de Minas Gerais, Belo Horizonte, Brazil.
E-mail: \href{mailto:gabriel@dcc.ufmg.br}{\texttt{gabriel@dcc.ufmg.br}}.}
\and Luis Higino\thanks{Departamento de Ci\^encia da Computa\c{c}\~ao,
Universidade Federal de Minas Gerais, Belo Horizonte, Brazil.
E-mail: \href{mailto:luis.higino@dcc.ufmg.br}{\texttt{luis.higino@dcc.ufmg.br}}.}}
\date{}

\begin{document}
\maketitle
\begin{abstract}
We give a matrix lifting of the entropy-splitting theorem of Csisz\'ar, K\"orner, Lov\'asz, Marton, and Simonyi, and show several applications. We use it to prove the Cartesian-product inequality conjectured by Sinjorgo and Sotirov for the generalized Lov\'asz theta number, and we also answer the saturation and convexification questions of Carli Silva, Coutinho, Oliveira, and Tun\c{c}el on the same theme. For the strong Hoffman bound, we answer a question of Wocjan, Elphick, and Darbari concerning orthogonal rank, and show that optimization over  matrix weights gives exactly the ceiling of the theta number of the complement. Finally, we extend the lifting to diagonally invariant cones of positive semidefinite matrices, giving a proof of a conjecture by Kuryatnikova, Sotirov, and Vera, and giving a short proof, for all graphs, of a theta-based weakening of Fajtlowicz's graph-energy conjecture.
\end{abstract}

\section{Lifted entropy splitting lemma}\label{sec:normalization}

Entropy splitting relates a convex corner to its antiblocker through coordinatewise products. Our starting point is that, for theta bodies, this relation can be lifted to their matrix certificates. The resulting matrix simultaneously satisfies a lower bound of rank one, an upper bound by the identity, and the zero pattern of a graph. These three properties connect several semidefinite and spectral graph parameters.

Throughout, $G=(V,E)$ is a finite simple graph with $n=|V|\geq1$. We write $\overline G$ for its complement, $\Smat^V$ for the real symmetric matrices indexed by $V$, and $\Smat_+^V$ for the positive semidefinite cone. We use $\one$ for the all-ones vector, $J=\one\one^T$, and $\ip{A}{B}=\tr(AB)$. The notations $\diag(A)$ and $\Diag(a)$ denote the diagonal vector extracting the diagonal of the matrix $A$ and the diagonal matrix with diagonal given by the vector $a$, respectively. If $p \in \R_+^V$, we write $\sqrt{p}$ for the vector obtained by coordinatewise square root. The product $A\circ B$ is the Schur product. Matrix inequalities $A\preceq B$ refer to the positive semidefinite order; $A\geq0$ means entrywise nonnegativity.

\subsection{Entropy splitting and theta bodies}

A \emph{convex corner} is a compact, convex, full-dimensional set $C\subseteq\R_+^V$ such that $0\leq z\leq x\in C$ implies $z\in C$. Its antiblocker is
\[
 \abl(C)=\{y\in\R_+^V:x^Ty\leq1\text{ for every }x\in C\}.
\]
The following is the entropy splitting theorem of Csisz\'ar, K\"orner, Lov\'asz, Marton, and Simonyi.

\begin{theorem}[{\cite[Theorem~1]{CsiszarEtAl1990}}]\label{thm:entropy} Let $C\subseteq\R_+^V$ be a convex corner. For every vector $p\in\R_+^V$ with $\one^Tp=1$, there are $x\in C$ and $y\in\abl(C)$ such that $x_i y_i=p_i$ for every $i\in V$.
\end{theorem}

The theta body of $G$ is
\begin{equation}\label{eq:theta-body}
 \ThetaBody(G)=\bigl\{x\in\R_+^V:\ \exists A\in\Smat^V,
  \ A\succeq xx^T,\ \diag A=x,\ A_{ij}=0\ (ij\in E)\bigr\}.
\end{equation}
It is a convex corner. We use the classical identities
\begin{equation}\label{eq:theta-duality}
 \vartheta(G;w)=\max_{x\in\ThetaBody(G)}w^Tx
 =\max\{\sqrt w^{\,T}X\sqrt w:X\succeq0,\ \tr X=1,
                  \ X_{ij}=0\ (ij\in E)\},
\end{equation}
for $w\geq0$, and
\begin{equation}\label{eq:antiblocking}
 \abl(\ThetaBody(G))=\ThetaBody(\overline G).
\end{equation}
Here $\vartheta(G)=\vartheta(G;\one)$ is the Lov\'asz theta number. These formulations and the antiblocking identity belong to the theta-body theory of Lov\'asz and Gr\"otschel--Lov\'asz--Schrijver \cite{Lovasz1979,GroetschelLovaszSchrijver1986}; see also \cite{Knuth1994,DeCarliSilvaTuncel2017}.

\subsection{The matrix lifting for the theta body}

In Section~\ref{sec:cone-lifting}, we show a more general version of the matrix lifting for convex corners. Here we present the specific case for the theta body, which will be used in several subsequent applications.

\begin{lemma}\label{lem:normalization}
Let $G=(V,E)$ and let $p\in\R_+^V$ satisfy $\one^Tp=1$. There is a matrix $R\in\Smat^V$ such that
\begin{equation}\label{eq:normalization}
 \sqrt p\,\sqrt p^{\,T}\preceq R\preceq I,\qquad
 R_{ij}=0\ (ij\in E),\qquad
 \diag R\in\ThetaBody(\overline G).
\end{equation}
In particular, $R\sqrt p=\sqrt p$ and $\tr R\leq\vartheta(\overline G)$.
\end{lemma}

\begin{proof}
We apply Theorem~\ref{thm:entropy} and \eqref{eq:antiblocking} to obtain $x\in\ThetaBody(G)$ and $y\in\ThetaBody(\overline G)$ with $x_i y_i=p_i$. Choose certificates $A$ and $B$ in~\eqref{eq:theta-body}, so that
\[
 A\succeq xx^T,\quad \diag A=x,\qquad
 B\succeq yy^T,\quad \diag B=y,
\]
and $A$ and $B$ vanish on the edges of $G$ and $\overline G$, respectively. Let $\dagger$ denote the Moore--Penrose pseudoinverse, and $\dagger/2$ denote the matrix square root of the pseudoinverse. Set
\[
  R = \Diag(p)^{\dagger/2}\Diag(y) A \Diag(y)\Diag(p)^{\dagger/2}.
\]
Then, by conjugating $\Diag(p)^{\dagger/2}\Diag(y)$, we obtain
\[
  A\succeq xx^T \implies R \succeq \sqrt p\,\sqrt p^{\,T}.
\]
The Schur product theorem gives $A\circ(B-yy^T)\succeq0$, and their complementary zero patterns give $A\circ B=\Diag(p)$. Thus,
\begin{equation}\label{eq:schur-step}
 A\circ(B-yy^T)\succeq0 \implies \Diag(p) \succeq \Diag(y)A\Diag(y),
\end{equation}
and by conjugating $\Diag(p)^{\dagger/2}$ we get $I \succeq R$.

Looking at $R$, the edge zeros from $A$ are preserved, and its non-zero diagonal entries are given by $R_{ii}=y_i^2x_i/p_i=y_i$. Thus $\diag R \leq y$, and therefore $\diag R\in\ThetaBody(\overline G)$.

Finally, for $v=\sqrt p$ we have $\|v\|=1$ and $vv^T\preceq R\preceq I$, so $v^T(I-R)v=0$. As $I-R\succeq0$, this implies $(I-R)v=0$. The trace bound follows from $\diag R\in\ThetaBody(\overline G)$.
\end{proof}

\subsection{Applications and organization}

Lemma~\ref{lem:normalization} has two useful features. Its eigenvector can be chosen to match either vertex weights or a leading eigenvector of a graph matrix, while its diagonal supplies an antiblocking certificate for a trace bound. Section~\ref{sec:generalized} uses these features for generalized theta numbers: Cartesian products, saturation, complementary products, and convexification. Section~\ref{sec:spectral} applies them to the strong Hoffman bound and identifies its optimum over signed matrix weights. Section~\ref{sec:stronger} replaces the graph zero pattern by a diagonally invariant cone. This yields the nonnegative generalized-theta comparison of Kuryatnikova, Sotirov, and Vera \cite{KuryatnikovaEtAl2022} and an all-graphs extension of a theta-based graph-energy inequality.

\section{Generalized Lov\'asz theta}\label{sec:generalized}

We first compare three generalizations of the theta number aimed at studying the maximum $k$-colourable subgraph problem. We prove the inequality in Sinjorgo and Sotirov's Cartesian-product conjecture; see \cite{SinjorgoSotirov2022}. We then answer the saturation question of Carli Silva, Coutinho, Oliveira, and Tun\c{c}el~\cite{CarliSilvaEtAl2026}, derive a complementary-product bound asked for by Sinjorgo and Sotirov~\cite{SinjorgoSotirov2022}, and identify the greatest convex minorant of the vertex-weighted generalized theta function.

\subsection{The generalized theta number and the \texorpdfstring{$k$}{k}-fold theta body}
\label{sec:objects}

Let $\alpha_k(G)$ be the maximum order of a $k$-colorable induced subgraph of $G$. Narasimhan and Manber~\cite{NarasimhanManber1990} introduced a generalization of the Lovász theta parameter that bounds $\alpha_k(G)$ from above. Its semidefinite formulation, due to Alizadeh~\cite[Section~5.3]{Alizadeh1995}, extends to real $0<k\leq n$ as
\begin{equation}\label{eq:nm}
 \vartheta_k(G)=\max\bigl\{\ip{J}{X}:
  0\preceq X\preceq I,\ \tr X=k,\ X_{ij}=0\ (ij\in E)\bigr\}.
\end{equation}
At $k=1$ this is $\vartheta(G)$. For integral $k$, a $k$-colorable induced subgraph with nonempty color classes $S_1,\ldots,S_k$ gives the feasible matrix $\sum_{a=1}^k |S_a|^{-1}\one_{S_a}\one_{S_a}^T$, and hence $\alpha_k(G)\leq\vartheta_k(G)$. The trace constraint can be presented with the inequality $\tr X \leq k$ with no change in the optimum value; see \cite[Lemma 1]{CarliSilvaCoutinhoGrandsire2021}.

Following~\cite{CarliSilvaEtAl2026}, define
\begin{equation}\label{eq:k-body}
 \ThetaBody_k(G)=k\ThetaBody(G)\cap[0,1]^V,\qquad \text{and} \quad 
 \theta_k(G)=\max_{z\in\ThetaBody_k(G)}\one^Tz.
\end{equation}
Thus $z\in\ThetaBody_k(G)$ exactly when $0\leq z\leq\one$ and there is a matrix $Z$ satisfying
\begin{equation}\label{eq:k-certificate}
 \diag Z=z,\qquad Z_{ij}=0\ (ij\in E),\qquad
 Z\succeq\frac1kzz^T.
\end{equation}
For integral $k$, $\theta_k(G)$ also bounds $\alpha_k(G)$: the incidence vector of a $k$-colorable vertex set, divided by $k$, lies in $\ThetaBody(G)$. The matrix-lifting relaxation with additional entrywise nonnegativity was studied by Kuryatnikova, Sotirov, and Vera~\cite{KuryatnikovaEtAl2022}. We treat that strengthening in Section~\ref{sec:ksv}.

A third bound is $\vartheta(G\cart K_k)$. The Cartesian product has vertex set $V\times\{1,\ldots,k\}$, with $(i,a)$ adjacent to $(j,b)$ when $a=b$ and $ij\in E$, or when $i=j$ and $a\ne b$. Since $\alpha(G\cart K_k)=\alpha_k(G)$, its theta number bounds the same problem. Sinjorgo and Sotirov~\cite[Conjecture~1]{SinjorgoSotirov2022} conjectured that
\[
 \vartheta(G\cart K_k)\leq\vartheta_k(G),
\]
with equality whenever $\vartheta_k(G)=k\vartheta(G)$. Carli Silva et al.~\cite{CarliSilvaEtAl2026} showed this conjecture to hold with equality in certain graphs admitting a strong degree of regularity: graphs whose adjacency matrix belongs to homogeneous partially coherent algebras. Here, we show the inequality for all graphs, and that the conjectured sufficient condition for equality does not hold in general.

\subsection{Cartesian products and a counterexample}\label{sec:comparison}

For real $k\geq1$, set
\begin{equation}\label{eq:q-sdp}
 q_k(G)=\max\left\{\ip{J}{Q}:
 \begin{array}{l}
 0\preceq Q\preceq kP,\quad \diag Q=\diag P,\quad \tr P=1,\\
 P,Q\in\Smat^V,\quad P_{ij}=0\ (ij\in E)
 \end{array}\right\}.
\end{equation}
For integral $k$, this equals $\vartheta(G\cart K_k)$; see the averaging argument introduced in~\cite[Lemma~2.8]{GvozdenovicLaurent2008} which leads directly to the formulation in~\cite[Equation~(9)]{CarliSilvaEtAl2026}. The constraints of this formulation, which extends naturally to real $k$, are given by
\[
 X-Y\succeq0,\quad X+(k-1)Y\succeq0,\quad
 X_{ij}=0\ (ij\in E),\quad\diag Y=0,\quad k\tr X=1,
\]
and the objective function to be maximized is $k\ip{J}{X}+k(k-1)\ip{J}{Y}$. The substitutions $P=kX$ and $Q=k(X+(k-1)Y)$ give~\eqref{eq:q-sdp}. At $k=1$, the constraints force $Q=P$.

The following theorem establishes the chain of inequalities relating $q_k(G)$, $\theta_k(G)$, and $\vartheta_k(G)$, and solves the inequality conjecture of Sinjorgo and Sotirov for all graphs. The inequality $q_k(G) \leq \theta_k(G)$ follows from elementary methods, and $\theta_k(G) \leq \vartheta_k(G)$ is a consequence of the lifted entropy splitting lemma.

\begin{theorem}\label{thm:chain}
For every graph $G$ on $n$ vertices and every real $1\leq k\leq n$,
\begin{equation}\label{eq:chain}
 q_k(G)\leq\theta_k(G)\leq\vartheta_k(G).
\end{equation}
Consequently, $\vartheta(G\cart K_k)\leq\vartheta_k(G)$ for every
integer $1\leq k\leq n$.
\end{theorem}

\begin{proof}
Let $(P,Q)$ be feasible in~\eqref{eq:q-sdp}, and suppose
$q=\ip{J}{Q}>0$. We will construct a feasible pair $(z,Z)$ for \eqref{eq:k-certificate} with objective value in \eqref{eq:k-body} at least $q$.

Let $e = \one/\sqrt q$. Put $d=\diag P=\diag Q$ and $a=Q e$. Define $D = \Diag(a) \Diag(d)^\dagger$, and
\[
 Z = D P D, \qquad \text{and} \quad z = \diag Z.
\]
The Cauchy--Schwarz inequality for the positive semidefinite form $Q$
gives, for any vector $x$, 
\begin{equation} \label{eq:CS}
 x^T aa^T x = (x^T Q e)^2 \leq x^T Q x \cdot e^T Q e = x^T Q x,
\end{equation}
hence $aa^T \preceq Q \preceq kP$, and thus
\[
 Z = DPD \succeq \frac{1}{k}(Da)(Da)^T = \frac{1}{k} zz^T.
\]
Also by \eqref{eq:CS} with $x$ the $i$th basis vector, we get $a_i^2 \leq Q_{ii} = d_i$, hence $0 \leq z_i = a_i^2/d_i \leq 1$ for all $i$ with $d_i>0$. Therefore $z\in\ThetaBody_k(G)$.

Since $\one^Td=1$, Cauchy--Schwarz now gives
\[
  q = (\one^Ta)^2 = (d^T \Diag(d)^\dagger a)^2 \leq (d^T \Diag(d)^\dagger d) (a^T \Diag(d)^\dagger a) = \one^T z \leq \theta_k(G).
\]
The case $q=0$ is immediate.

For the second inequality, take $z\in\ThetaBody_k(G)$ with $s=\one^Tz>0$.
Apply Lemma~\ref{lem:normalization} to $p=z/s$ and let
$M=\Diag(\sqrt z)R\Diag(\sqrt z)$.
Then $0\preceq M\preceq\Diag(z)\preceq I$ and $M$ vanishes on the edges.Antiblocking duality gives
\[
 \tr M=z^T\diag R\leq k,
 \qquad
 \ip{J}{M}=s\sqrt p^{\,T}R\sqrt p=s.
\]
The trace-inequality version of~\eqref{eq:nm} now gives
$\vartheta_k(G)\geq\one^Tz$. The case $z=0$ is trivial.
\end{proof}

The equality clause in Sinjorgo and Sotirov's conjecture fails for
a disjoint union of two complete graphs, as seen below.

\begin{proposition}\label{prop:counterexample}
Let $G=K_6\mathbin{\dot\cup}K_2$ and $k=3$. Then
\[
 \vartheta_3(G)=3\vartheta(G)=6,
 \qquad \vartheta(G\cart K_3)=5.
\]
\end{proposition}

\begin{proof}
The theta number is additive on disjoint unions, so $\vartheta(G)=2$. Consider
\[
 X=\begin{pmatrix}
     \tfrac14 I_6&\tfrac18 J_{6,2}\\
     \tfrac18 J_{2,6}&\tfrac34 I_2
    \end{pmatrix}.
\]
The subspaces orthogonal to the all-ones vectors in the two blocks give eigenvalues $1/4$ with multiplicity five and $3/4$ with multiplicity one. On the remaining two-dimensional subspace, the eigenvalues are $0$ and $1$. Thus $0\preceq X\preceq I$, $\tr X=3$, and $X$ has the required edge zeros. Its objective is $\ip{J}{X}=6$. Conversely, $\vartheta_3(G)\leq3\vartheta(G)=6$, by dividing a feasible matrix by three in the trace-one formulation of theta.

The graph $K_a\cart K_b$ has independence number and theta number $\min\{a,b\}$: a matching of pairs gives the lower bound, and a cover by $\min\{a,b\}$ cliques gives the upper bound. Additivity therefore gives
\[
 \vartheta(G\cart K_3)
 =\vartheta(K_6\cart K_3)+\vartheta(K_2\cart K_3)=3+2=5.
 \qedhere
\]
\end{proof}

\subsection{Saturation parameters}\label{sec:mu}

The generalized theta number yields a lower bound on the chromatic number through its failure to reach $n$. Mohar and Poljak~\cite[Corollary~4.16]{MoharPoljak1993} observed that
\[
 \vartheta_k(G)<n\quad\Longrightarrow\quad\chi(G)>k
 \qquad(k\text{ integral}).
\]
Allowing real values, Carli Silva et al.~\cite[Section~2.2]{CarliSilvaEtAl2026} define the saturation parameter
\begin{equation}\label{eq:mu-def}
 \mu(G)=\min\{k\in[1,n]:\vartheta_k(G)=n\}.
\end{equation}
It satisfies $\mu(G)\leq\chi(G)$, and $\lceil\mu(G)\rceil$ is the integer saturation parameter of Sinjorgo and Sotirov~\cite[Equation~(39)]{SinjorgoSotirov2022}. The equivalent semidefinite formulation is
\begin{equation}\label{eq:mu-sdp}
 \mu(G)=\min\{\tr R:0\preceq R\preceq I,
                 \ R_{ij}=0\ (ij\in E),\ R\one=\one\}.
\end{equation}
Indeed, for a contraction $R$, the equality $\ip{J}{R}=n$ holds if and only if $(I-R)\one=0$.

We answer~\cite[Question~(iii)]{CarliSilvaEtAl2026}.

\begin{theorem}\label{thm:mu}
For every graph $G$,
\[
 \mu(G)\leq\vartheta(\overline G).
\]
\end{theorem}

\begin{proof}
Apply Lemma~\ref{lem:normalization} with $p=\one/n$. The resulting matrix is feasible in~\eqref{eq:mu-sdp} and satisfies $\tr R\leq\vartheta(\overline G)$.
\end{proof}

We have found that Silvestri~\cite{Silvestri2017} had already shown $\mu(G) \leq \vartheta^+(\overline G)$, where $\vartheta^+$ is Szegedy's variant~\cite{Szegedy1994}. Recalling that $\vartheta \leq \vartheta^+$ and this inequality can be strict, our result strengthens Silvestri's bound in general. Their proof uses diagonal scaling~\cite{BrualdiParterSchneider1966} to make $R=DXD$ stochastic, but for that $X$ must be entrywise nonnegative and hence applies only to $\vartheta^+$. Lemma~\ref{lem:normalization} removes this restriction by providing the contraction directly, without requiring entrywise nonnegativity.

The prescribed eigenvector also gives an exact worst-case form of the
saturation bound. Set
\begin{equation}\label{eq:weighted-saturation}
 \mu(G,p)=\min\{\tr R:\sqrt p \sqrt p^T\preceq R\preceq I,
                            \ R_{ij}=0\ (ij\in E)\}.
\end{equation}
For positive semidefinite contractions, $R \sqrt p = \sqrt p$ is equivalent to
$R\succeq \sqrt p \sqrt p^T$, so this is the least trace at which the objective
$\sqrt p^T R \sqrt p$ reaches one. In particular, $\mu(G,\one/n)=\mu(G)$.

\begin{proposition}\label{prop:worst-saturation}
For every graph $G$,
\[
 \max_{p\geq0,\ \one^Tp=1}\mu(G,p)=\vartheta(\overline G).
\]
One maximizing $p$ is the diagonal of an optimal trace-one theta matrix
for $\overline G$.
\end{proposition}

\begin{proof}
The upper bound is Lemma~\ref{lem:normalization}. For the reverse bound, choose $Y\succeq0$ with $\tr Y=1$, zero entries on $E(\overline G)$, and $\ip{J}{Y}=\vartheta(\overline G)$. Set $p=\diag Y$. On $S=\supp p$, let $C_{ij}=Y_{ij}/\sqrt{p_i p_j}$; extend $C$ by an identity matrix on $V\setminus S$ and zero cross entries. Then $C\succeq0$, $\diag C=\one$, and its off-diagonal support is contained in $E(G)$. Every $R$ feasible in~\eqref{eq:weighted-saturation} therefore satisfies
\[
 \tr R=\ip{C}{R}\geq \sqrt p^T C \sqrt p
       =\ip{J}{Y}=\vartheta(\overline G).
 \qedhere
\]
\end{proof}

\subsection{The complementary-product question}\label{sec:product}

The classical inequality
$\vartheta(G)\vartheta(\overline G)\geq n$ motivates the question of
Sinjorgo and Sotirov~\cite[Section~9]{SinjorgoSotirov2022} about a
corresponding lower bound for generalized theta numbers. The saturation
certificate gives a sharp bound for every real $k\in[1,n]$.

\begin{proposition}\label{thm:product}
For every graph $G$ on $n$ vertices and every real $1\leq k\leq n$,
\begin{equation}\label{eq:complement-product}
 \vartheta_k(G)\vartheta_k(\overline G)\geq nk.
\end{equation}
The bound is attained by $G=K_n$ and by $G=\overline{K_n}$.
\end{proposition}

\begin{proof}
Lemma~\ref{lem:normalization}, with $p=\one/n$, gives a contraction $R$ with the edge zeros of $G$, $\tr R\leq \vartheta(\overline G)$, and $\ip{J}{R}=n$. Scaling $R$ by $\min\{1,k/\vartheta(\overline G)\}$ shows that
\begin{equation}\label{eq:product-factor}
 \vartheta_k(G)\geq n\min\{1,k/\vartheta(\overline G)\}.
\end{equation}
On the other hand, \cite[Lemma 1]{CarliSilvaCoutinhoGrandsire2021} and the feasible matrix $(k/n)I$ give $\vartheta_k(\overline G)\geq\max\{\vartheta(\overline G),k\}$. Multiplying proves~\eqref{eq:complement-product}. For $K_n$, the two factors are $k$ and $n$.
\end{proof}

\subsection{Vertex weights and convexification}\label{sec:weights}

For $w\in\R_+^V$ and real $0<k\leq n$, define
\begin{align}
 \vartheta_k(G;w)
 &=\max\{\sqrt w^{\,T}X\sqrt w:0\preceq X\preceq I,
             \ \tr X=k,\ X_{ij}=0\ (ij\in E)\},\label{eq:weighted-nm}\\
 \theta_k(G;w)&=\max\{w^Tz:z\in\ThetaBody_k(G)\}.\label{eq:weighted-body}
\end{align}
At $k=1$ both equal $\vartheta(G;w)$. The function $\vartheta_k(G;\cdot)$ is continuous and positively homogeneous, but need not be convex~\cite[Section~5.2]{CarliSilvaEtAl2026}.

\begin{theorem}\label{thm:weighted-domination}
For every graph $G$, every real $0<k\leq n$, and every $w\geq0$,
\[
 \theta_k(G;w)\leq\vartheta_k(G;w).
\]
\end{theorem}

\begin{proof}
Take $z\in\ThetaBody_k(G)$ with $s=w^Tz>0$ and apply Lemma~\ref{lem:normalization} to $p_i=w_i z_i/s$. As in Theorem~\ref{thm:chain}, the matrix
$M=\Diag(\sqrt z)R\Diag(\sqrt z)$ is a contraction with the required edge zeros and $\tr M=z^T\diag R\leq k$. Its objective is
\[
 \sqrt w^{\,T}M\sqrt w=s\sqrt p^{\,T}R\sqrt p=s.
\]
The case $s=0$ is immediate. Maximizing over $z$ proves the claim.
\end{proof}

For a nonnegative positively homogeneous function $f$ on $\R_+^V$, its greatest convex minorant is
\begin{equation}\label{eq:conv-decomposition}
 f^{\mathrm{conv}}(w)=\inf\left\{\sum_{r=1}^{m}f(w^{(r)}):
 m\geq1,\ w^{(r)}\geq0,\ \sum_{r=1}^{m}w^{(r)}=w\right\}.
\end{equation}
The right-hand side is subadditive, positively homogeneous, and at most $f$. Conversely, if $g$ is convex and $g\leq f$, every such decomposition gives $g(w)\leq m^{-1}\sum_r g(mw^{(r)})\leq\sum_r f(w^{(r)})$. For $f(w)=\vartheta_k(G;w)$, denote this minorant by $\vartheta_k^{\mathrm{conv}}(G;w)$.
It is also the gauge of the convex hull of $\{w\geq0:\vartheta_k(G;w)\leq1\}$. In particular note that this set is compact, since it is closed, and bounded by $\vartheta_k(G;w)\geq(k/n)\one^Tw$.

We record the support-function identity behind the convexification.
Its proof applies to any convex corner, a fact used again in
Section~\ref{sec:ksv}.

\begin{lemma}\label{lem:splitting}
Let $C\subseteq\R_+^V$ be a convex corner, and let
$h_C(a)=\max_{x\in C}a^Tx$. For $k>0$ and $w\geq0$,
\begin{equation}\label{eq:weight-splitting-general}
 h_{kC\cap[0,1]^V}(w)
 =\min_{0\leq a\leq w}\{k h_C(a)+\one^T(w-a)\}.
\end{equation}
\end{lemma}

\begin{proof}
Let $Q=\{z\in\R^V:z\leq\one\}$. Since $kC$ contains an interior
point strictly below $\one$, the relative interiors of $kC$ and $Q$
intersect. The support-function formula for intersections
\cite[Corollary~16.4.1]{Rockafellar1970} therefore gives
\[
 h_{kC\cap[0,1]^V}(w)
 =\min_{a+s=w}\{h_{kC}(a)+h_Q(s)\}
 =\min_{s\geq0}\{k h_C(w-s)+\one^Ts\}.
\]
Here we used $kC\subseteq\R_+^V$ and the fact that
$h_Q(s)=\one^Ts$ for $s\geq0$, whereas $h_Q(s)=+\infty$
otherwise.

Because $C$ is closed under coordinatewise decreases,
$h_C(v)=h_C(v_+)$: coordinates with negative coefficients can
be set to zero without leaving $C$. Consequently, replacing
$s$ by $s'$, where $s'_i=\min\{s_i,w_i\}$, preserves the support
term, since
\[
 w-s'=(w-s)_+,
 \qquad h_C(w-s')=h_C(w-s),
\]
and cannot increase $\one^Ts$. We may therefore restrict the
minimum to $0\leq s\leq w$. Substituting $a=w-s$ gives
\[
 h_{kC\cap[0,1]^V}(w)
 =\min_{0\leq a\leq w}
   \{k h_C(a)+\one^T(w-a)\},
\]
as claimed.
\end{proof}

\begin{theorem}\label{thm:convexification}
For every graph $G$, real $0<k\leq n$, and $w\geq0$,
\begin{equation}\label{eq:convexification}
 \vartheta_k^{\mathrm{conv}}(G;w)=\theta_k(G;w)
 =\min_{0\leq a\leq w}
       \{k\vartheta(G;a)+\one^T(w-a)\}.
\end{equation}
\end{theorem}

\begin{proof}
The support function $\theta_k(G;\cdot)$ is convex and, by Theorem~\ref{thm:weighted-domination}, is dominated by $\vartheta_k(G;\cdot)$. Thus $\theta_k(G;w)\leq\vartheta_k^{\mathrm{conv}}(G;w)$. For the reverse inequality, subadditivity and minorization give, for $0\leq a\leq w$,
\[
 \vartheta_k^{\mathrm{conv}}(G;w)
 \leq\vartheta_k(G;a)+\vartheta_k(G;w-a)
 \leq k\vartheta(G;a)+\one^T(w-a).
\]
The last two bounds follow from dividing a feasible matrix by $k$ and from $X\preceq I$, respectively. Minimize over $a$ and apply Lemma~\ref{lem:splitting} with $C=\ThetaBody(G)$.
\end{proof}

The theorem answers~\cite[Question~(i)]{CarliSilvaEtAl2026}. It also
identifies the threshold after convexification. Set
\[
 K_{\mathrm{conv}}(G)=\inf\{k\in[1,n]:
                  \vartheta_k^{\mathrm{conv}}(G;\one)=n\}.
\]

\begin{corollary}\label{cor:conv-saturation}
For every graph $G$, $K_{\mathrm{conv}}(G)=\vartheta(\overline G)$.
\end{corollary}

\begin{proof}
Compactness and Theorem~\ref{thm:convexification} give
\[
 \vartheta_k^{\mathrm{conv}}(G;\one)=n
 \iff \one\in\ThetaBody_k(G)
 \iff \one/k\in\ThetaBody(G)
 \iff \vartheta(\overline G)\leq k,
\]
where the last equivalence is antiblocking duality.
\end{proof}

\section{Graph eigenvalues and the chromatic number}\label{sec:spectral}

We next apply the lifting to a leading eigenvector. It produces a contraction whose trace controls the number of smallest eigenvalues needed to offset the largest one. This proves the strong Hoffman bound, see \cite{Hoffman1970}, with $\lceil\vartheta(\overline G)\rceil$ in place of the chromatic number and answers the orthogonal-rank question of Wocjan, Elphick, and Darbari~\cite{WocjanElphickDarbari2020} for this bound. Extending the argument to matrix weights and combining it with a result of Carli Silva, Coutinho, and Grandsire~\cite{CarliSilvaCoutinhoGrandsire2021} identifies the optimized Hoffman bound exactly.

\subsection{Hoffman bounds and related results}\label{sec:hoffman-survey}

For a nonzero real symmetric matrix $W$ of trace zero, write its
eigenvalues as $\lambda_1(W)\geq\cdots\geq\lambda_n(W)$ and define
\begin{equation}\label{eq:kappa}
 \kappa(W)=\min\left\{r\in\{1,\ldots,n-1\}:
       \lambda_1(W)+\sum_{i=1}^r\lambda_{n+1-i}(W)\leq0\right\}.
\end{equation}
The set is nonempty because the expression is zero for $r=n-1$. We put $\kappa(0)=0$. For the adjacency matrix $A$ of $G$, write $\kappa(G)=\kappa(A)$.
Hoffman's bounds~\cite{Hoffman1970} state, for a graph with an edge,
\begin{equation}\label{eq:hoffman}
 1+\frac{\lambda_1(A)}{-\lambda_n(A)}\leq\chi(G),
 \qquad 1+\kappa(G)\leq\chi(G).
\end{equation}
We refer to the second inequality as the strong Hoffman bound. It is at least as strong as the ceiling of the first.

Wocjan, Elphick, and Darbari~\cite[Theorem~7]{WocjanElphickDarbari2020} proved $1+\kappa(G)\leq\chi_q(G)$, where $\chi_q$ is the quantum chromatic number. In their concluding questions, they asked whether the orthogonal rank $\xi(G)$ could also replace $\chi_q(G)$. Here $\xi(G)$ is the smallest $d$ for which nonzero vectors in $\C^d$ can be assigned to vertices so that adjacent vertices receive orthogonal vectors. For further spectral bounds on orthogonal and projective rank, see Wocjan and Elphick~\cite{WocjanElphick2019}. Recent work of Abiad and Meeus~\cite{AbiadMeeus2025} develops related Hoffman bounds for distance colorings, quantum colorings, and vector colorings.

To discuss optimization, let
\[
 \mathcal W(G)=\{W\in\Smat^V:\diag W=0,
       \ W_{ij}=0\text{ if }i\ne j\text{ and }ij\notin E\}.
\]
Carli Silva, Coutinho, and Grandsire~\cite{CarliSilvaCoutinhoGrandsire2021} studied
\begin{equation}\label{eq:h-def}
 h(G)=\max_{W\in\mathcal W(G)}(1+\kappa(W))
\end{equation}
and established, in particular,
\begin{equation}\label{eq:ccg-lower}
 \lceil\vartheta(\overline G)\rceil\leq h(G).
\end{equation}
Their comparison uses the weighted Hoffman characterization of theta,
\begin{equation}\label{eq:theta-ratio}
 \vartheta(\overline G)
 =\max_{0\ne W\in\mathcal W(G)}
        \left(1+\frac{\lambda_1(W)}{-\lambda_n(W)}\right),
\end{equation}
valid when $G$ has an edge; see also~\cite{Lovasz1979,Knuth1994}.
We prove the reverse inequality in~\eqref{eq:ccg-lower}.

\subsection{A theta bound for the strong Hoffman parameter}
\label{sec:strong-hoffman}

It is useful to retain the fractional part of the last eigenvalue used in~\eqref{eq:kappa}. For a nonzero symmetric trace-zero matrix $W$, let $r=\kappa(W)$, put
$b_i=-\lambda_{n+1-i}(W)$ for $1\leq i\leq r$, and define 
\begin{equation}\label{eq:rho}
 \rho(W)=r-1+
       \frac{\lambda_1(W)-\sum_{i=1}^{r-1}b_i}{b_r}.
\end{equation}
The minimality of $r$ implies $b_r>0$ and
$r-1<\rho(W)\leq r$, so $\lceil\rho(W)\rceil=\kappa(W)$.
Set $\rho(0)=0$.

\begin{lemma}\label{lem:spectral-trace}
For every nonzero real symmetric matrix $W$ of trace zero,
\begin{equation}\label{eq:rho-sdp}
 \rho(W)=\min\{\tr Q:0\preceq Q\preceq I,
                              \ip{W}{Q}\leq-\lambda_1(W)\}.
\end{equation}
\end{lemma}

\begin{proof}
In an orthonormal eigenbasis of $W$, the diagonal entries of $Q$ lie in $[0,1]$ and determine both its trace and $\ip{W}{Q}$. Replacing $Q$ by its diagonal preserves feasibility and both quantities. The resulting linear program minimizes $\sum_i q_i$ subject to $0\leq q_i\leq1$ and $\sum_i\lambda_i(W)q_i\leq-\lambda_1(W)$. An optimum uses the most negative eigenvalues first: it takes the first $r-1$ of them with coefficient one and the next with the fraction in \eqref{eq:rho}. This gives~\eqref{eq:rho-sdp}.
\end{proof}

\begin{theorem}\label{thm:strong-hoffman}
For the adjacency matrix $A$ of every graph $G$,
\begin{equation}\label{eq:rho-theta}
 1+\rho(A)\leq\vartheta(\overline G),
 \qquad
 1+\kappa(G)\leq\lceil\vartheta(\overline G)\rceil.
\end{equation}
\end{theorem}

\begin{proof}
The result is immediate for an edgeless graph. Otherwise, choose a nonnegative unit eigenvector $u$ of $A$ for $\lambda_1(A)$ and apply Lemma~\ref{lem:normalization} to $p_i=u_i^2$. The matrix $Q=R-uu^T$ satisfies $0\preceq Q\preceq I$ and
\[
 \ip{A}{Q}=\ip{A}{R}-u^TAu=-\lambda_1(A),
 \qquad \tr Q=\tr R-1\leq\vartheta(\overline G)-1.
\]
Here $\ip{A}{R}=0$ follows from the edge zeros of $R$. Lemma~\ref{lem:spectral-trace} proves the first inequality. Taking ceilings gives the second.
\end{proof}

The standard bound $\vartheta(\overline G)\leq\xi(G)$ now answers the question of Wocjan, Elphick, and Darbari in \cite{WocjanElphickDarbari2020}.

\begin{corollary}\label{cor:orthogonal}
For every graph $G$, $1+\kappa(G)\leq\xi(G)$.
\end{corollary}

\subsection{Optimization over matrix weights}\label{sec:weighted-hoffman}

The argument extends to matrix weights by a diagonal change of
signs in the lifting.

\begin{theorem}\label{thm:weighted-hoffman}
For every $W\in\mathcal W(G)$,
\[
 1+\rho(W)\leq\vartheta(\overline G),
 \qquad 1+\kappa(W)\leq\lceil\vartheta(\overline G)\rceil.
\]
Consequently,
\begin{equation}\label{eq:optimized-hoffman}
 h(G)=\lceil\vartheta(\overline G)\rceil.
\end{equation}
\end{theorem}

\begin{proof}
The case $W=0$ is immediate. Otherwise, take a real unit eigenvector $u$ for $\lambda_1(W)$. Apply Lemma~\ref{lem:normalization} with $p_i=u_i^2$ to obtain $R_0$, and let $S$ be diagonal with entries in $\{-1,1\}$ chosen so that $S\sqrt p=u$. Then $R=SR_0S$ satisfies $uu^T\preceq R\preceq I$, has the edge zeros of $G$, and has the same diagonal as $R_0$. For $Q=R-uu^T$ we therefore have
\[
 0\preceq Q\preceq I,\qquad
 \ip{W}{Q}=-\lambda_1(W),\qquad
 \tr Q\leq\vartheta(\overline G)-1.
\]
Lemma~\ref{lem:spectral-trace} gives the inequalities. Maximizing over $W$ and using \eqref{eq:ccg-lower} proves~\eqref{eq:optimized-hoffman}.
\end{proof}

The fractional refinement is exact as well.

\begin{corollary}\label{cor:optimized-fractional}
For every graph $G$,
\[
 \max_{W\in\mathcal W(G)}(1+\rho(W))=\vartheta(\overline G).
\]
\end{corollary}

\begin{proof}
Theorem~\ref{thm:weighted-hoffman} gives the upper bound. For $W\ne0$, formula~\eqref{eq:rho-sdp} implies $\rho(W)\geq\lambda_1(W)/(-\lambda_n(W))$. Equation~\eqref{eq:theta-ratio} gives the reverse bound when $G$ has an edge. Both sides are one for an edgeless graph.
\end{proof}

\section{A stronger lifting lemma}\label{sec:stronger}

The proof of Lemma~\ref{lem:normalization} used two complementary matrix certificates and the Schur product theorem. A different proof works for every closed cone of positive semidefinite matrices preserved by nonnegative diagonal congruence. We formulate it in the PSD-lifting framework of Carli Silva and Tun\c{c}el~\cite{DeCarliSilvaTuncel2017}, then apply it to two different problems.

\subsection{PSD liftings and diagonal congruence}\label{sec:cone-lifting}

Let $\cK\subseteq\Smat_+^V$ be a closed convex cone containing all
nonnegative diagonal matrices. Assume that
\begin{equation}\label{eq:diagonal-invariance}
 X\in\cK,\ d\in\R_+^V
 \quad\Longrightarrow\quad \Diag(d)X\Diag(d)\in\cK.
\end{equation}
Following the PSD-lifting framework
of~\cite{DeCarliSilvaTuncel2017}, define
\[
 \Psd(\cK)=\left\{
 \begin{pmatrix}t&x^T\\x&X\end{pmatrix}\succeq0:X\in\cK\right\},
\]
and let its theta corner be
\begin{equation}\label{eq:cone-corner}
 \cC_{\cK}=\{x\in\R_+^V:\exists X\in\cK,
                         \ \diag X=x,\ X\succeq xx^T\}.
\end{equation}
This is $\ThetaBody(\Smat^V,\Psd(\cK))$ in their notation.
It is a convex corner containing the standard simplex. Its support
function has the trace-one representation
\begin{equation}\label{eq:cone-support}
 h_{\cC_{\cK}}(w)
   =\max\{\sqrt w^{\,T}X\sqrt w:X\in\cK,\ \tr X=1\}
   \qquad(w\geq0).
\end{equation}
These facts follow from the PSD-lifting and support-function results
in~\cite[Corollary~2 and Theorem~5]{DeCarliSilvaTuncel2017}.

We use the following elementary fact about copositive matrices. If $T$ is symmetric, $a^T T a \geq0$ for all $a\geq0$, and $v>0$ satisfies $v^T T v=0$, then $(v+\varepsilon z)^T T(v+\varepsilon z)\geq0$ for every $z$ and all sufficiently small positive and negative $\varepsilon$. Its linear coefficient must vanish, and then $z^T T z\geq0$ for every $z$.

\begin{lemma}\label{lem:cone-normalization}
Let $\cK$ be a closed convex cone of positive semidefinite matrices containing all
nonnegative diagonal matrices. Assume $\cK$ is diagonally scaling-invariant, as in \eqref{eq:diagonal-invariance}. For every vector $p\in\R_+^V$ with $\one^Tp=1$, there exists $R\in\cK$ such that 
\begin{equation}\label{eq:cone-normalization}
 \sqrt p\,\sqrt p^{\,T}\preceq R\preceq I,
 \qquad \diag R\in\abl(\cC_{\cK}).
\end{equation}
In particular, $R\sqrt p=\sqrt p$.
\end{lemma}

\begin{proof}
First let $p>0$. By Theorem~\ref{thm:entropy}, choose positive $x\in\cC_{\cK}$ and $y\in\abl(\cC_{\cK})$ with $x_i y_i=p_i$. Take $X\in\cK$ with $\diag X=x$ and $X\succeq xx^T$, and put
\[
 D=\Diag(\sqrt{y/x}),\qquad R=DXD,\qquad v=\sqrt p.
\]
Then, as before, $R\in\cK$, $R\succeq vv^T$, and $\diag R=y$. It remains to prove $R\preceq I$.

Since $x^Ty=1$, the support value $h_{\cC_{\cK}}(y)$ is one. For any $a\geq0$, let $T=\Diag(a/\sqrt y)$ and $B=TRT\in\cK$. We have $\tr B=\|a\|^2$. Equation~\eqref{eq:cone-support}, applied to $B/\tr B$ when $a\ne0$, gives
\[
 a^TRa=\sqrt y^{\,T}B\sqrt y\leq\tr B=\|a\|^2.
\]
Thus $I-R$ is copositive. Because $R\succeq vv^T$ and $\|v\|=1$, we have $v^T(I-R)v\leq0$, while copositivity gives the reverse inequality. The positive-zero observation above implies $I-R\succeq0$.

For a probability vector with zero coordinates, replace $p$ by $(p+\varepsilon\one)/(1+n\varepsilon)$ and take a convergent subsequence. Closedness of $\cK$ and $\abl(\cC_{\cK})$, together with compactness of the contraction interval, preserves the claims. The eigenvector assertion follows as in Lemma~\ref{lem:normalization}.
\end{proof}

We spell out three graph cones:
\begin{align*}
 \cK_0(G)&=\{X\succeq0:X_{ij}=0\ (ij\in E)\},\\
 \cK_-(G)&=\{X\succeq0:X\geq0,\ X_{ij}=0\ (ij\in E)\},\\
 \cK_+(G)&=\{X\succeq0:X_{ij}\leq0\ (ij\in E)\}.
\end{align*}
Their corners are denoted by $\ThetaBody(G)$, $\ThetaBody^-(G)$, and $\ThetaBody^+(G)$,
respectively. Their support functions at $w\geq0$ are
$\vartheta(G;w)$, $\vartheta^-(G;w)$, and $\vartheta^+(G;w)$.
The latter two are the Schrijver and Szegedy variants
\cite{Schrijver1979,Szegedy1994}. The antiblocking identities are
\begin{equation}\label{eq:plus-minus-duality}
 \abl(\ThetaBody^-(G))=\ThetaBody^+(\overline G),\qquad
 \abl(\ThetaBody^+(G))=\ThetaBody^-(\overline G);
\end{equation}
see~\cite{DeCarliSilvaTuncel2017}.
Consequently, the stronger theorem gives either a nonnegative
contraction with edge zeros and diagonal in $\ThetaBody^+(\overline G)$,
or a contraction with nonpositive edge entries and diagonal in
$\ThetaBody^-(\overline G)$.

\subsection{The Kuryatnikova--Sotirov--Vera conjecture}\label{sec:ksv}

Kuryatnikova, Sotirov, and Vera~\cite[Conjecture~1]{KuryatnikovaEtAl2022} compared two semidefinite relaxations with entrywise nonnegativity for the maximum $k$-colorable subgraph problem. In the notation below, their conjecture is the unweighted inequality $\theta_k^-(G;\one)\leq\vartheta_k^-(G;\one)$.
For real $0<k\leq n$ and $w\geq0$, define
\begin{align}
 \vartheta_k^-(G;w)
 &=\max\{\sqrt w^{\,T}X\sqrt w:
       X\in\cK_-(G),\ X\preceq I,\ \tr X=k\},
       \label{eq:nonnegative-nm}\\
 \theta_k^-(G;w)
 &=\max\{w^Tz:z\in k\ThetaBody^-(G)\cap[0,1]^V\}.
       \label{eq:nonnegative-body}
\end{align}
In particular, feasibility in the second formulation means that $0\leq z\leq\one$ and there exists $Z\in\cK_-(G)$ with $\diag Z=z$ and $Z\succeq zz^T/k$.
These are the two relaxations in their conjecture, with $\theta_k^-$ corresponding to their matrix-lifting formulation $\theta_k^3$. As before, one may replace $\tr X=k$ by $\tr X\leq k$ in~\eqref{eq:nonnegative-nm}.

\begin{theorem}\label{thm:ksv}
For every graph $G$, every real $0<k\leq n$, and every $w\geq0$,
\begin{equation}\label{eq:ksv-comparison}
 \theta_k^-(G;w)\leq\vartheta_k^-(G;w).
\end{equation}
Moreover, the greatest convex minorant in the vertex weights satisfies
\begin{equation}\label{eq:ksv-convexification}
 (\vartheta_k^-)^{\mathrm{conv}}(G;w)
 =\theta_k^-(G;w)
 =\min_{0\leq a\leq w}
        \{k\vartheta^-(G;a)+\one^T(w-a)\}.
\end{equation}
\end{theorem}

\begin{proof}
Let $z\in k\ThetaBody^-(G)\cap[0,1]^V$ with $s=w^Tz>0$. Apply Lemma~\ref{lem:cone-normalization} to $\cK_-(G)$ and $p_i=w_i z_i/s$. The resulting $R$ is entrywise nonnegative, has the edge zeros of $G$, and has diagonal in $\ThetaBody^+(\overline G)$. Set $M=\Diag(\sqrt z)R\Diag(\sqrt z)$. Then
\[
 M\in\cK_-(G),\qquad M\preceq\Diag(z)\preceq I,\qquad
 \tr M=z^T\diag R\leq k,
\]
where the last inequality follows from~\eqref{eq:plus-minus-duality}. Also $\sqrt w^{\,T}M\sqrt w=s$. This proves \eqref{eq:ksv-comparison}; the case $s=0$ is immediate.

For the convexification, use Lemma~\ref{lem:splitting} with
$C=\ThetaBody^-(G)$. The proof of Theorem~\ref{thm:convexification} then
applies verbatim: the support function is a convex minorant by the
comparison just proved, and
\[
 \vartheta_k^-(G;a)\leq k\vartheta^-(G;a),\qquad
 \vartheta_k^-(G;w-a)\leq\one^T(w-a).
\]
These inequalities give the reverse comparison with the greatest
convex minorant.
\end{proof}

Thus the conjectured ordering holds with arbitrary nonnegative vertex weights and real trace budgets. The smaller relaxation also has an exact interpretation: it is the convexification of the larger one as a function of the vertex weights.

\subsection{Graph energy and a weakening of Fajtlowicz's conjecture}
\label{sec:energy}

The energy of a graph with adjacency matrix $A$ is $\cE(G)=\sum_i|\lambda_i(A)|$.
Since $\tr A=0$, its half-energy is the sum of the positive eigenvalues, and hence
\begin{equation}\label{eq:energy-sdp}
 \frac{\cE(G)}2=\max\{\ip{A}{X}:0\preceq X\preceq I\}.
\end{equation}
Fajtlowicz conjectured that
\begin{equation}\label{eq:fajtlowicz}
 \frac{\cE(G)}2\geq n-\alpha(G).
\end{equation}
Abiad, Coutinho, Juliano, and Reijnders~\cite[Theorem~1.6]{AbiadEtAl2026} used semidefinite programming to study this conjecture and proved, under their high-regularity hypothesis, the weaker inequality $\cE(G)/2\geq n-\vartheta^-(G)$.
The stronger lifting gives this latter inequality for every graph.

\begin{theorem}\label{thm:energy}
For every graph $G$ on $n$ vertices,
\begin{equation}\label{eq:energy-bound}
 \frac{\cE(G)}2\geq n-\vartheta^-(G).
\end{equation}
\end{theorem}

\begin{proof}
Apply Lemma~\ref{lem:cone-normalization} to $\cK_+(\overline G)$ and $p=\one/n$. It gives a positive semidefinite contraction $R$ with $R\one=\one$, nonpositive entries on $E(\overline G)$, and $\diag R\in\ThetaBody^-(G)$. If $\overline A$ is the adjacency matrix of $\overline G$, then
\[
 \ip{A}{R}
 =\ip{J-I-\overline A}{R}
 =n-\tr R-\ip{\overline A}{R}
 \geq n-\vartheta^-(G).
\]
Use $R$ in~\eqref{eq:energy-sdp}.
\end{proof}

The full inequality~\eqref{eq:fajtlowicz} has been proved in the arXiv preprint of Kumar and Pragada~\cite[Theorem~1.2]{KumarPragada2026}. Since $\alpha(G)\leq\vartheta^-(G)$, their result implies Theorem~\ref{thm:energy}. The contribution here is the short semidefinite proof and the removal of the regularity assumption from the theta-based argument of~\cite{AbiadEtAl2026}.

\section*{Acknowledgments}

The authors acknowledge support from the Funda\c{c}\~ao de Amparo
\`a Pesquisa do Estado de Minas Gerais (FAPEMIG) and the Conselho
Nacional de Desenvolvimento Cient\'ifico e Tecnol\'ogico (CNPq).

\section*{Statement on the use of artificial intelligence}

A direct prompt to Claude Opus 5.5 resulted in a somewhat convoluted solution to the saturation-parameter question treated in Section~\ref{sec:mu}. The authors worked on distilling the key steps and found the combination of entropy splitting and matrix lifting presented in Sections~\ref{sec:normalization} and \ref{sec:stronger}. The authors then proceeded to find the other main applications in this paper. ChatGPT Astra 6 was used to help with the writing of the paper, but the final version was prepared by the authors. It was also responsible for the counterexample in Section~\ref{sec:comparison}, presented as Proposition~\ref{prop:counterexample}. The authors retain full responsibility for the content of the paper.

{\small
\bibliographystyle{abbrvnat}
\bibliography{lifted_entropy_splitting}
}
\end{document}